\documentclass[11pt]{amsart}
\usepackage{amssymb,amsfonts,amsmath}
\usepackage[all]{xy}
\usepackage[shortlabels]{enumitem}

\usepackage[usenames,dvipsnames]{pstricks}
\usepackage{wrapfig}
\usepackage{epsfig}
\usepackage{pst-grad} 
\usepackage{pst-plot} 

\usepackage{euscript,color}

\newtheorem{theorem}{Theorem}
\newtheorem{proposition}{Proposition}[section]
\newtheorem{question}{Question}[section]
\newtheorem{corollary}{Corollary}
\theoremstyle{definition}

\newtheorem{lemma}{Lemma}[section]
\newtheorem{example}{Example}[section]
\theoremstyle{remark}
\newtheorem{remark}{Remark}
\theoremstyle{claim}
\newtheorem{claim}{Claim}[section]

\theoremstyle{theoremA}

\theoremstyle{theoremB}

\newcommand{\real}{\mathbb{R}}

\newcommand{\Q}{\mathbb{Q}^n_c}
\newcommand{\hy}{\mathbf{H}}

\newcommand{\Si}{\Sigma}

\newcommand{\te}{\theta}

\newcommand{\al}{\alpha}
\newcommand{\om}{\omega}

\newcommand{\na}{\nabla}
\newcommand{\m}{\mathcal}
\newcommand{\ep}{\epsilon}
\newcommand{\ga}{\gamma}

\newcommand{\be}{\beta}
\newcommand{\de}{\delta}
\newcommand{\ti}{\tilde}

\newcommand{\p}{\partial}

\newcommand{\g}{\gamma}

\newcommand{\lf}{\left}
\newcommand{\rg}{\right}
\newcommand{\fr}{\frac}

\newcommand{\pr}{\pi_{_W}}

\title[Isometry actions and geodesics orthogonal to submanifolds]{Isometry actions and geodesics orthogonal to submanifolds}
\author[A. Di Scala]{Antonio J. Di Scala}
\address{Dipartimento di Matematica,
Politecnico di Torino,
Corso Duca degli Abruzzi 24, 10129 Torino, Italy}
\email{antonio.discala@polito.it}
\author[S. Mendon\c ca]{S\'ergio Mendon\c ca}
\address{Departamento de An\'alise, Instituto de Matem\'atica,
Universidade Federal Fluminense, Niter\'oi, RJ, CEP 24020-140,
Brasil} \email{sergiomendonca@id.uff.br}
\author[H. Mirandola]{Heudson Mirandola}
\address{Departamento de M\'etodos Matem\'aticos, Instituto de Matem\'atica, Universidade Fe\-deral do Rio de Janeiro, Rio de Janeiro, RJ, CEP 21945-970, Brasil} \email{heudson@im.ufrj.br}
\thanks{This work is partially supported by CNPq.}
\author[G. Ruiz-Hern\'andez]{Gabriel Ruiz-Hern\'andez}
\address{Instituto de Matem\'aticas, Universidad Nacional Autonoma de M\'exico, Ciudad Universitaria,
C.P. 04510 Mexico City, D.F., Mexico}
\email{gruiz@matem.unam.mx}
\begin{document}
\maketitle
\section{\bf Introduction}
A simple well-known fact says that if $f:\Si \to \real^{n}$
is
 an immersion satisfying that at each point of $f(\Si)$ there exists a normal
 line intersecting a fixed point $p\in\real^n$ then $f(\Si)$ is contained in a
 round sphere centered at $p$. In this paper we will provide two generalizations
 of this fact, obtaining also an application to horospheres in Hadamard manifolds.

To state our first result let us fix some notations. For an arbitrary subset $C$ of a Riemannian manifold $M$  and $r\ge 0$ we set
  $$\mathcal S(C,r)= \left\{x\in M \bigm| d(x,C)=r\right\},$$
where $d$ is
 the distance function.

 We will denote by $\Q$ the complete
 simply-connected $n$-dimensional manifold of constant curvature $c$.
 Let $W=W^j$ denote a complete connected $j$-dimensional totally geodesic submanifold of $\Q$.
 If $c\le 0$ there exists a natural projection $\pi_{_W}:\Q\to W$ satisfying
 $\pr(q)=\g(1)$, where $\g:[0,1]\to \Q$ is the unique geodesic with $\g(0)=q$, $\g(1)\in W$ and
 the lenght $L(\g)=d(q,W)$.
 Now we recall how this projection may be defined in the case $c>0$. We first set
 $$V_W=\mathcal S\Big(W,\fr{\pi}{2\sqrt{c}}\Big).$$ It is well known that $V_W$ is a totally geodesic sphere of dimension $n-j-1$. To construct
 a natural projection $\pr: (\Q-V_W)\to W$    we consider the normal bundle
 $$\nu(W)=\lf\{(x,v)\bigm|x\in W,\, v\in (T_xW)^\perp\rg\},$$
 where $(T_xW)^\perp$ denotes the orthogonal complement of $T_xW$ relatively to $T_x(\Q)$.
 Let $\exp^\perp:\nu(W)\to \Q$ denote the normal exponential map. Set
 $$B_W=\lf\{(x,v)\in\nu(W)\bigm||v|<\fr\pi{2\sqrt{c}}\rg\}.$$
 It is well known that the map $\exp^\perp|_{B_W}:B_W\to (\Q-V_W)$ is a diffeomorphism and that
 $\exp^\perp(\partial B_W)=V_W$, where $\partial B_W$ denotes the boundary of the closure $\bar B_W$.   Thus we may define the projection $\pr: (\Q-V_W)\to W$  by $\pr(\exp^\perp(p,v))=p$.
In other words, for $q\in \Q-V_W$ it holds that $\pr(q)=\g(1)$, where $\g:[0,1]\to \Q$ is the unique geodesic with $\g(0)=q$, $\g(1)\in W$ and
$L(\g)=d(q,W)$.

We denote by $G_W$ the group of isometries of $\Q$ that fix each point in $W$.
Given a tangent vector $v$ in some point in a Riemannian  manifold,
 we will denote by $\g_{v}$ a
 geodesic satisfying $\g_{v}'(0)=v$. The domain of $\g_v$ will be specified in each case.\\

Let $\Si\subset \Q$ be a connected embedded submanifold of the space
form $\Q$  of class $C^k$, with $k\ge 1$. Let $W$ be a complete connected
totally geodesic submanifold of $\Q$. We will consider the following properties:

\begin{enumerate}[(A)]
\item \label{invariant} For each point $q\in\Si$ there exists a neighborhood $U$ of $q$ in $\Si$
such that $U$ is contained in an embedded $C^k$ hypersurface of $\Q$ which is invariant
under the action of $G_W$.\\
\item \label{orthogonal} For any point $q\in \Si$, there exists a vector $\eta\in T_q(\Q)$ orthogonal to $\Si$
such that the geodesic $\g_\eta$ intersects $W$.\\
\item \label{vector} For any point $q\in \Si$ and any vector $v\in T_q\Si$ with $(d\pr)_qv=0$, there exists
a vector $\eta\in T_q(\Q)$ orthogonal to $v$
such that the geodesic $\g_\eta$ intersects $W$.
\end{enumerate}

\begin{theorem} \label{revolution}
Under the above notations assume that $\Si\cap W=\emptyset$ and the map $\pr|_\Si:\Si\to W$ is a submersion.
In the case $c>0$ we assume further that $\Si \cap V_W=\emptyset$. Then it holds that:
\begin{enumerate}[(i)]
\item If $c=0$ then properties {\rm \ref{invariant}}, {\rm \ref{orthogonal}} and {\rm \ref{vector}} are equivalent.
\item \label{inv-ort} If $c>0$ then {\rm \ref{invariant}} and {\rm \ref{orthogonal}} are equivalent;
\item\label{vector-invariant} If $c<0$ then {\rm \ref{invariant}} and {\rm \ref{vector}} are equivalent.
\end{enumerate}
\end{theorem}

\begin{remark}\label{imply} Since {\rm \ref{orthogonal}} implies {\rm \ref{vector}} trivially,
we obtain from Theorem \ref{revolution} that {\rm \ref{orthogonal}} implies {\rm \ref{invariant}} and {\rm \ref{invariant}} implies {\rm \ref{vector}} for all values of $c$.
\end{remark}
\begin{remark}\label{c} It is simple to show that  \ref{vector} is always true if $c>0$ (see 
Proposition \ref{true}).  
\end{remark}
\begin{remark}\label{sharp}
Theorem \ref{revolution} is sharp  in the sense that all the implications that do not appear in Theorem \ref{revolution} or in Remark \ref{imply} fail (see Section \ref{examples}). We will also see in Section \ref{examples} that the assumption
that $\pi_{_W}|_{\Si}$ is a submersion may not be dropped. We will also see in Proposition \ref{bsubmersion} that if $c\le 0$ and $\Si$ is a
hypersurface of $\Q$, then property {\rm \ref{vector}} implies that $\pi_{_W}|_\Sigma$ is a submersion.
\end{remark}

Example \ref{complex} presents a nontrivial situation in
$\real^4$ in which Theorem \ref{revolution} holds. In this example, if $p\in W-\{(0,0,0,0),(0,0,1,0)\}$ then the complete totally geodesic submanifold of maximal dimension which is orthogonal to $W$ at $p$ intersects $\Si$ in infinitely many isolated points.

\begin{question} If we remove from Theorem \ref{revolution} the assumption that $\pi_{_W}|_\Si$ is a submersion, 
is it true that \ref{orthogonal} implies \ref{invariant} in an open dense subset of 
$\Si$? 
\end{question}
In the next results we will relax the $C^1$-hypothesis and just consider a differentiable
immersion. Let $M$ be a Hadamard manifold. It is well known that $M$ admits a natural compactification $\bar M= M \cup M({\infty}),$ where the ideal boundary $M(\infty)$ consists of the asymptotic classes $\g(\infty)$ of geodesic rays $\g$ in $M$ (see \cite{e-on} or Chapter 3 of \cite{b-g-s}). We obtained
the following result (compare with condition \ref{vector} in Theorem \ref{revolution}).

\begin{theorem}\label{hadamard} Let $f:\Si\to M$ be a differentiable immersion of a connected manifold $\Si$ in a Hadamard manifold $M$. Fix $x_0\in M({\infty})$ and assume that for all point $p\in \Si$ and $v\in T_p\Si$ there exists a vector $\eta\in T_{f(p)}M$ orthogonal to $df_pv$ such that the geodesic ray $\ga_\eta:[0,+\infty)\to M$ satisfies that $\g_\eta(\infty)=x_0$.
Then $f(\Si)$ is contained in a horosphere of $M$ associated with $x_0$.
\end{theorem}

The above result can be proved by using the following general result.

\begin{theorem}\label{cylinder} Let $f:\Si\to M$ be a differentiable immersion of a connected manifold $\Si$ in a Riemannian manifold $M$. Let $G:M\to \real$ be a Lipschitz function with Lipschitz constant $C>0$.
Assume that for all $p\in \Si$  and any $v\in T_p\Si$ there exists a nontrivial vector $\eta\in T_{f(p)}M$ orthogonal to $df_pv$ such that the geodesic $\ga_\eta:[0,1]\to M$ satisfies that $$|G(f(p))-G(\ga_\eta(1))|=C\, L(\ga_\eta).$$
Then $f(\Si)$ is contained in a level set of $G$.
\end{theorem}

Given an arbitrary subset $\m A$ of a manifold $M$ the distance function from $\m A$ is Lipschitz with Lipschitz constant $1$ and vanishes on $\m A$. Thus we may apply Theorem \ref{cylinder} to obtain the following result.

\begin{corollary}\label{cylinder2} Let $f:\Si\to M$ be a differentiable immersion of a connected manifold $\Si$ in a Riemannian manifold $M$. Let $\m A\subset M$ be an arbitrary subset. Assume that for all $p\in \Si$  and $v\in T_p\Si$ there exists a vector $\eta\in T_{f(p)}M$ orthogonal to $df_pv$ such that the geodesic $\ga_\eta:[0,1]\to M$ satisfies that $\g_\eta(1)\in \m A$ and $$L\lf(\ga_\eta\rg)=d(f(p),\m A).$$
Then $f(\Si)$ is contained in $\mathcal S(\m A,r)$ for some constant $r\ge 0$.
\end{corollary}

\section{\bf Isometry actions and submanifolds}
The purpose of this section is to prove Theorem \ref{revolution}, which will be done after
stating some lemmas. The first one is a well known simple result about the geometry of manifolds
with constant sectional curvature.
\begin{lemma} \label{triangle} Let $(\g_1,\g_2,\g_3)$ be a geodesic
triangle where each $\g_i:[a_i,b_i]\to \Q$ is a minimal geodesic. Then there
exists a totally geodesic surface $N^2\subset \Q$ which is
isometric to $\Bbb Q_c^2$ and contains the images of $\g_1, \g_2$ and $\g_3$.
\end{lemma}

Let $W\subset \Q$ be a complete totally geodesic connected submanifold and fix $p\in W$.
It is well known that there exists a unique complete totally geodesic connected submanifold
$S=S_{p{W}}\subset\Q$ containing $p$ such that the tangent space $T_pS$ agrees
with the orthogonal complement $(T_pW)^\perp$. The following result is well known and
follows easily from the equality $\pi_{_W}(\exp^\perp(x,\om))=x$, where $(x,\om)\in \nu(W)$ for
 all $c\in\real$, and satisfies $|\om|<\frac{\pi}{2\sqrt{c}}$ if $c>0$.

\begin{lemma} \label{piw} With the notations above, the map $\pr$ is a submersion on its domain.
Furthermore for $p\in W$ it holds that $\pr^{-1}(\{p\})=S_{p{W}}$ in the case $c\le 0$
and $\pr^{-1}(\{p\})=S_{p{W}}\cap(\Q-V_W)$ in the case $c>0$. In particular for $q\in \pr^{-1}(\{p\})$ the kernel ${\rm{Ker}}\lf((d\pi_{_W})_q\rg)=T_q(S_{pW})$.
\end{lemma}

The next lemma is a simple consequence of Lemma \ref{triangle}.
\begin{lemma} \label{trivial} For $c>0$, let $\g:\real\to
\Q$ be a normal geodesic. Let  $\al:[0,t_0]\to\Q$ and $\be:[0,s_0]\to\Q$ be minimal normal geodesics
satisfying that:
\begin{enumerate}[(i)]
\item $\al(0)=\g(a)$ and $\be(0)=\g(b)$ with $0\le a\le b\le\fr\pi{\sqrt{c}}$;
\item \label{orthogonallity} $\lf<\al'(0),\g'(a)\rg>=\lf<\be'(0),\g'(b)\rg>=0$;
\item $\al(t_0)=\be(s_0)$.
\end{enumerate}
Then it holds that $\be'(0)$ is the parallel transport of $\al'(0)$ along $\g$.
\end{lemma}
\begin{proof} By Lemma \ref{triangle} there exists a
totally geodesic surface $N^2\subset \Q$ which is
isometric to $\Bbb Q_c^2$ containing the images of $\g, \al$ and $\be$.
 It is not difficult to conclude that either $b-a=\fr\pi{\sqrt{c}}$ and the union of the images of
 $\al$ and $\be$ determines a geodesic arc of length $\fr\pi{\sqrt{c}}$, or $b-a<\fr\pi{\sqrt{c}}$
 and $s_0=t_0=\fr\pi{2\sqrt{c}}$. In both cases the conclusion of Lemma \ref{trivial} holds.
 \end{proof}

\begin{lemma} \label{pi2} Let $\Si\subset \Q$ be a differentiable embedded connected submanifold
with $c>0$. Let $W$ be a closed connected totally geodesic submanifold
of $\Q$ and fix a point $q\in \Si\cap (\Q-\{W\cup V_W\})$. Assume
that the map $\pr|_\Si:\Si\to W$ is a submersion at $q$ and that
there exists a vector $\eta\in T_q(\Q)$ orthogonal to $\Si$
such that the geodesic $\g_\eta$ intersects $W$.
Consider a shortest normal geodesic $\g:[0,r_0]\to \Q$ from $q$ to $W$, namely, assume that
$\g(0)=q$, $\g(r_0)=p=\pr(q)\in W$ and $L(\g)=r_0=d(q,W)$. Then it holds that
$\lf<\eta,\g'(0)\rg>\not=0.$
\end{lemma}\begin{proof} Consider the totally geodesic sphere $S=S_{pW}$ as in Lemma \ref{piw}.
Since $\g'(r_0)\in (T_pW)^\perp=T_pS$ it follows that
the image of $\g$ is contained in $S$, hence $q\in S$.

Since $\g_\eta$ intersects $W$
and $q=\g_\eta(0)\notin W$ we have easily that $\eta\not=0$. Without loss of
generality we will assume that $|\eta|=1$.
The intersection between the image of $\g_\eta$ and $W$
occurs in two antipodal points, hence there exists $0<s_0<\fr{\pi}{\sqrt{c}}$
such that $u=\g_\eta(s_0)\in W$.

Now we assume by contradiction that $\lf<\g_\eta'(0),\g'(0)\rg>=\lf<\eta,\g'(0)\rg>=0.$
This fact and the inequalities $0<r_0<\fr{\pi}{2\sqrt{c}}$ and
$0<s_0<\fr{\pi}{\sqrt{c}}$ imply together that $p\not=u$.
Thus there exists a minimal normal geodesic $\mu:[0,t_0]\to W$
satisfying that $\mu(0)=p$ and $\mu(t_0)=u$. Since $\mu'(0)\in T_pW$ we have that $\mu'(0)$
is orthogonal to $\g'(r_0)$.
Since $u=\g_\eta(s_0)=\mu(t_0)$ and $\mu$ and $\g_\eta$ are minimal normal geodesics orthogonal to
$\g$,
we may apply Lemma \ref{trivial} to conclude that $\mu'(0)$ is the parallel transport of $\eta$ along $\g$.

Write $T_q\Si=(T_qS\cap T_q\Si)\oplus V$ and set $j=\dim(W)$ the dimension of $W$. Since $V\subset T_q\Si$ we have
that $V\cap T_qS=V\cap(T_qS\cap T_q\Si)=\{0\}$. Since $\pi_{_W}|_\Si$ is a submersion, we have that $(d\pr)_q(T_q\Si)=T_pW$, hence it follows from
Lemma \ref{piw} that $\dim(V)\ge j.$

Let $P:T_q(\Q)\to T_p(\Q)$ be the parallel transport
along $\g$. Since $V\subset T_q\Si$ we obtain that $\eta$ is orthogonal to the linear space $V$.
Since $P(\eta)=\mu'(0)$ we obtain that $\mu'(0)$ is orthogonal to the image $P(V)$. Since
$\mu'(0)\in T_pW$ it must be orthogonal to $T_pS$. Thus we have that
$\mu'(0)$ is orthogonal to $(P(V)+T_pS)$. Furthermore it holds that
$$P(V)\cap T_pS=P(V)\cap P(T_qS)=P(V\cap T_qS)=\{0\}.$$
We conclude that $$\dim(P(V)+T_pS)=\dim(P(V))+\dim(T_pS)\ge j+(n-j)=n,$$ hence
$P(V)+T_pS=T_p(\Q)$ and $\mu'(0)=0$. This contradicts the fact that  $|\mu'(0)|=1$. Lemma \ref{pi2}
 is proved.
\end{proof}

\begin{proof}[\bf Proof of Theorem \ref{revolution}] Let $W$ and $\Si$ be submanifolds of $\Q$ satisfying
the hypotheses of Theorem \ref{revolution}.
Our first goal is to prove that Property \ref{invariant} holds if one
of the following conditions hold:
\begin{enumerate}[(I)]
\item \label{I} $c>0$ and Property \ref{orthogonal}  holds;
\item \label {II} $c\le 0$ and Property \ref{vector} holds.
\end{enumerate}
Thus we will assume that \ref{I} or \ref{II} holds and we will prove that each sufficiently small open subset of $\Si$ is contained in a hypersurface invariant under the action of $G_W$.

Fix $q\in \Si$.
Consider a normal shortest geodesic $\g:[0,r_0]\to \Q$ from $q$ to $W$, namely, assume that
$\g(0)=q$, $\g(r_0)=p=\pr(q)\in W$ and $L(\g)=d(q,W)=r_0$.
Set $S=S_{pW}$. Since $\g'(r_0)\in S$ it follows that
the image of $\g$ is contained in $S$, hence $q\in S$.

Fix $v\in T_q\Si$ with $(d\pi_{_W})_qv=0$. By Lemma \ref{piw} we have that
$v\in T_qS$.
\begin{claim} \label{sphere} $\lf<v,\g'(0)\rg>=0$.
\end{claim}
In fact, by using \ref{I} or \ref{II}, we may choose $\eta\in T_q\Q$ such that the geodesic $\g_\eta$ intersects $W$ and one of the
following properties holds: (a) $\eta$ is orthogonal to $T_q\Si$ and $c>0$; (b) $\eta$ is orthogonal to $v$ and $c\le 0$. Recall that $\eta\not=0$ since
$q\notin W$ and $\g_\eta(\real)$ intersects $W$. Without loss of generality we
will assume that $|\eta|=1$. If $\eta$ and $\g'(0)$ are linearly dependent Claim \ref{sphere}
follows trivially. Thus we may assume that $\eta$ and $\g'(0)$ are linearly independent.

In the case $c\le 0$ the intersection between $\g_\eta$ and $W$ occurs at a unique
point $u=\g_\eta(s_0)\in W$. If $c>0$ the intersection between $\g_\eta$ and $W$
occurs in two antipodal points, hence there exists $0<s_0<\fr{\pi}{\sqrt{c}}$
such that $u=\g_\eta(s_0)\in W$. In both cases the geodesic $\g_\eta:[0,s_0]\to
\Q$ is the unique minimal normal geodesic joining $q$ and $u$. We
have that $p\not=u$ because of the two following facts: (i) $\g$ and $\g_\eta$ are the unique
minimal normal geodesics from $q$ to $p$ and $q$ to $u$, respectively; (ii) $\eta$ and $\g'(0)$
are linearly independent. Thus we obtain that
there exists a minimal normal geodesic $\mu:[0,t_0]\to W$ with $t_0>0$, 
satisfying that $\mu(0)=p$, $\mu(t_0)=u$.

Now we assert that
\begin{equation}\label{perp}\lf<\eta,\g'(0)\rg>\not=0.
\end{equation}
In fact, if $c\le 0$ and (\ref{perp}) is false, the lines $\g_\eta$ and
$\mu$ are mutually orthogonal to $\g$ which implies that they cannot
intersect in the point $u$, which is a contradiction.  In the case $c>0$, the assertion (\ref{perp}) follows
from Lemma \ref{pi2}.

Let $P:T_q(\Q)\to T_p(\Q)$ be the parallel transport along
$\g$. We claim that
\begin{equation}\label{li}
P(\eta) \rm{\ and\ } \mu'(0) \rm{\ are\ linearly\ independent.\ }
\end{equation}
In fact, if (\ref{li}) is not true we have that $P(\eta)=\pm\mu'(0)$. Since $\mu'(0)\in T_pW$
and $\g'(r_0)\in T_pS$ it holds that $\lf<\mu'(0),\g'(r_0)\rg>=0$, hence  we
have that $$\lf<\eta,\g'(0)\rg>=\lf<P(\eta),P(\g'(0))\rg>=\lf<P(\eta),\g'(r_0)\rg>=\pm\lf<\mu'(0),\g'(r_0)\rg>=0,$$
which contradicts (\ref{perp}).

Now we assert that
\begin{equation}\label{pmu} \lf<P(v),\mu'(0)\rg>=\lf<P(v),P(\eta)\rg>=0.
\end{equation}
The equality $\lf<P(v),P(\eta)\rg>=0$ follows directly from the equality $\lf<v,\eta\rg>=0$, which follows from (a) or (b).
Since $v\in T_qS$ and $S$ is totally geodesic we obtain that $P(v)\in T_pS=(T_pW)^\perp$. This implies that $\lf<P(v),\mu'(0)\rg>=0$.

By Lemma \ref{triangle} there exists a complete totally geodesic surface $N^2$ containing
the images of $\g$, $\g_\eta$ and $\mu$. Since $\eta\in T_q(N^2)$ and $N^2$ is
totally geodesic it follows that $P(\eta)\in T_p(N^2)$. Thus (\ref{li}) implies that $P(\eta)$ and $\mu'(0)$
form a basis for $T_p(N^2)$. From (\ref{pmu}) we obtain that $P(v)$ is orthogonal to
$T_p(N^2)$, which implies that
$$\lf<v,\g'(0)\rg>=\lf<P(v),P(\g'(0))\rg>=\lf<P(v),\g'(r_0)\rg>=0,$$
since $\g'(r_0)\in T_p(N^2).$ Claim \ref{sphere} is proved.

Now we are in position to prove that \ref{invariant}  holds
under condition \ref{I} or \ref{II} above. To do this we fix $q\in \Si$. Since $\pi_{_W}|_\Si$
is a submersion and $\Si$ is of class $C^k$ with $k\ge 1$, there exists a $C^k$ diffeomorphism $h:D\times \m V\to \m U$ satisfying that  $\pi_{_W}(h(x,y))=y$ for any $(x,y)\in D\times \m V$,
where $\m U\subset \Si$ is a small open neighborhood of $q$, $\m V=\pi_{_W}(\m U)$,
  and $D$ is an open disk in $\real^{m-j}$ with $m=\dim (\Si)$ and $j=\dim(W)$. If
  $c>0$ then $\m U$ may be chosen sufficiently small so that $\m U\cap V_W=\emptyset$.

Write $q=h(x_q,p)$, where $\pi_{_W}(q)=p$. Define the $C^k$ map $\xi:\m V\to \m U$ given by $$\xi(y)=h(x_q,y).$$

\begin{claim} \label{constant} For any $y\in \m V$ and $z, \ti z\in h(D\times\{y\})$,
it holds that $d(z,W)=d(\ti z,W)$.
\end{claim}
In fact, for any $x\in D$, we have that $\pi_{_W}(h(x,y))=y$, hence $\pi_{_W}(u)=y$ for any
$u\in h(D\times\{y\})$. Thus any vector $v$ tangent to $h(D\times\{y\})$ in $u$ must satisfies
that $(d\pi_{_W})_uv=0$. By Claim \ref{sphere} it holds that $\lf<v,\g'(0)\rg>=0$ where
$\g:[0,r_0]\to \Q$ is the normal shortest geodesic from $u$ to $W$, namely, it satisfies that $\g(0)=u, \g(r_0)=y$ and $L(\g)=r_0=d(u,W)$. Thus we may apply Corollary \ref{cylinder2} to conclude that
$d(z,W)=d(\ti z, W)$ for all $z, \ti z\in h(D\times\{y\})$. Claim \ref{constant} is proved.

Given $z\in \m U$, it holds that $z=h(x,\pi_{_W}(z))$ for some $x\in D$. We also have that $\xi(\pi_{_W}(z))=h(x_q,\pi_{_W}(z))$. Thus we obtain that $z$ and $\xi(\pi_{_W}(z))$ belong
to $h(D\times\{\pi_{_W}(z)\})$. Thus we conclude from Claim \ref{constant} that
\begin{equation} \label{distance}
d(z, W)=d(\xi(\pi_{_W}(z)), W).
\end{equation}

We define the $C^k$ function $r:\m V\to (0,\infty)$ given by $r(y)=d(\xi(y),W)$. Consider the following set
$$M=\bigcup_{y\in \m V}S'(y,r(y)),$$
where $S'(y,s)$ denotes the sphere on $S_{yW}$ of center $y$ and radius $s$.
\begin{claim} \label{M} The set $M$ is invariant under the action of the group $G_W$.
\end{claim}
In fact, fix an isometry $\phi\in G_W$ and $y\in \m V$. For $w\in T_yW$ and $v\in T_y(S_{yW})=(T_yW)^
\perp$ we have
that
$$\lf<d\phi_yv,w\rg>=\lf<d\phi_yv,d\phi_yw\rg>=\lf<v,w\rg>=0,$$
hence $d\phi_y(T_y(S_{yW}))\subset (T_yW)^
\perp=T_y(S_{yW})$ and, by an argument on dimension we conclude that
$d\phi_y(T_y(S_{yW}))=T_y(S_{yW})$. From this and the fact that
 $S_{yW}$ and $\phi(S_{yW})$ are totally geodesic it follows that $\phi(S_{yW})=S_{yW}$.
 Furthermore we observe that the distance relatively to $S_{yW}$ agrees with the distance on $\Q$, since $S_{yW}$ is totally geodesic. This implies that $d(u,y)=r(y)$ for all $u\in S'(y,r(y))$, hence
 $S'(y,r(y))\subset \m S(y,r(y))$. This
 together with the fact that $\phi(S_{yW})=S_{yW}$ leads us to the conclusion that
 $$\phi(S'(y,r(y)))\subset \m S(y,r(y))\cap S_{yW}=S'(y,r(y)).$$
  Claim \ref{M} is proved.

\begin{claim} \label{subset} The set $M$ contains $\m U$.
\end{claim}
In fact, take $z\in \m U$. Set $y=\pi_{_W}(z)$.
To prove Claim \ref{subset} it suffices to prove that $z\in S'(y,r(y))$. Clearly we have that $z\in S_{yW}$. By (\ref{distance})
we obtain that
$$d(z,y)=d(z,\pi_{_W}(z))=d(z,W)=d(\xi(\pi_{_W}(z)),W)=d(\xi(y),W)=r(y).$$
Claim \ref{subset} is proved.
\begin{claim}\label{embedded} The set $M$ is an embedded hypersurface of class $C^k$.
\end{claim}
In fact, let $\nu_1(\m V)=\{(y,v)\bigm|y\in\m V, \ v\in (T_y\m V)^\perp \mbox{ with } |v|=1\}$ denote the unit normal fiber bundle over $\m V$. We define
the $C^k$ map $\psi:\nu_1(\m V)\to \Q$ given by
$$\psi(y,v)=\exp^\perp(y,r(y)v)$$
and $\varphi:M\to \nu_1(\m V)$ given by
$$\varphi(z)=\lf(\pi_1\lf(\lf(\exp^\perp\rg)^{-1}(z)\rg),\fr{\pi_2\lf(\lf(\exp^\perp\rg)^{-1}(z)\rg)}
{\lf|\pi_2\lf(\lf(\exp^\perp\rg)^{-1}(z)\rg)\rg|}\rg),$$
where $\pi_1$ and $\pi_2$ are the natural projections given by $\pi_1(y,v)=y$ and $\pi_2(y,v)=v$.
It is clear that $\psi((y,v))\subset S_{yW}$ and $d(\psi((y,v)),y)=r(y)$, hence
 we have that $\psi((y,v))\in S'(y,r(y))$. Thus we obtain that $\psi(\nu_1(\m V))\subset M$.
 Furthermore we have that $\varphi$ is the restriction of a $C^\infty$
map defined in $\Q-W$ in the case $c\le 0$ and defined in $\Q-(W\cup V_W)$ in
the case $c>0$.

It is straightforward to show that $\varphi(\psi(y,v))=(y,v)$. We will show that
$\psi(\varphi(z))=z$. Set $$y=\pi_1\lf(\lf(\exp^\perp\rg)^{-1}(z)\rg) \ \mbox{ and } \ v=\fr{\pi_2\lf(\lf(\exp^\perp\rg)^{-1}(z)\rg)}{\lf|\pi_2\lf(\lf(\exp^\perp\rg)^{-1}(z)\rg)\rg|}.$$
With this notation we have that
$\varphi(z)=(y,v)$.
Note that $$\pi_{_W}(z)=\pi_1\lf(\lf(\exp^\perp\rg)^{-1}(z)\rg)=y.$$ By (\ref{distance}) we have that
\begin{equation*}
\lf|\pi_2\lf(\lf(\exp^\perp\rg)^{-1}(z)\rg)\rg|=d(z,W)=d(\xi(\pi_{_W}(z)),W)=r(\pi_{_W}(z))=r(y),
\end{equation*}
which implies that $r(y)v=\pi_2\lf(\lf(\exp^\perp\rg)^{-1}(z)\rg)$. Thus we have that
\begin{eqnarray*}
\psi(\varphi(z))&=& \psi(y,v)=\exp^\perp(y,r(y)v)=\exp^\perp\lf(\pi_1\lf(\lf(\exp^\perp\rg)^{-1}(z)\rg), r(y)v\rg)\\&=& \exp^\perp\lf(\pi_1\lf(\lf(\exp^\perp\rg)^{-1}(z)\rg), \pi_2\lf(\lf(\exp^\perp\rg)^{-1}(z)\rg)\rg)\\&=& z.
\end{eqnarray*}
We conclude that $M=\psi(\nu_1(\m V))$ and $\psi$ is a $C^k$-diffeomorphism, hence $M$ is an embedded hypersurface of $\Q$ of class $C^k$. Claim \ref{embedded} is proved.

It follows from Claim \ref{M}, Claim \ref{subset} and Claim \ref{embedded} that Property \ref{invariant} holds. 

To finishes the proof of Theorem \ref{revolution} we need to prove that \ref{invariant} implies \ref{orthogonal} in the case $c\ge 0$ and that \ref{invariant} implies \ref{vector} in the 
case $c<0$. 

From now on let us assume that \ref{invariant} holds. We fix $q\in \Si$
and take a small neighborhood $\m U$ of $q$ in $\Si$ contained in an embedded $C^1$ hypersurface $M$ in $\Q$ that is invariant under the action of $G_W$.

Let $\g:[0,r_0]\to \Q$ be the normal shortest geodesic from $q$ to $W$, namely, assume that $\g(0)=q,\  \g(r_0)=p\in W$ and $L(\g)=r_0=d(q,W)$, hence $p=\pi_{_W}(q)=\g(r_0)$.  Set $S=S_{pW}$. Let
$S'=S'(p,r_0)$ be the sphere on $S$ of center $p$ and radius $r_0$. Since $M$ is invariant under the action of $G_W$ it is not difficult to see that $S'\subset M$. 

\begin{claim} \label{glinha} $\g'(0)\notin T_qM$.
\end{claim}
In fact, since $\pi_{_W}|_{\Si}$ is a submersion we obtain that $$T_pW=d(\pi_{_W})_q (T_q(\Si))\subset d(\pi_{_W})_q (T_q M),$$ hence $d(\pi_{_W})_q(T_qM)=T_qW$. Furthermore we have that
\begin{equation}\label{kernel} T_q(S')\subset T_qS\cap T_qM={\rm{Ker}}(d\pi_{_W})_q\cap T_qM={\rm{Ker}}(d(\pi_{_W}|_M)_q).\end{equation}
We obtain that
$\dim(\rm{Ker}(d(\pi_{_W}|_M)_q))=\dim(M)-\dim(W)=n-1-j,$ where $j=\dim(W)$. Since $\dim(S')=\dim(S)-1=n-j-1$, we obtain from (\ref{kernel}) that 
\begin{equation}\label{kernelM}
\rm{Ker}(d(\pi_{_W}|_M)_q))=T_q(S').
\end{equation}
Now assume by contradiction that  $\g'(0)\in T_qM$. Since $\g'(0)\in \rm{Ker}(d\pi_{_W})_q$ and
 is orthogonal to $T_qS'$ we have that $\dim(\rm{Ker}(d(\pi_{_W}|_M)_q))\ge 1+\dim(S')$, which
is a contradiction. Claim \ref{glinha} is proved.
\begin{claim}\label{AimpliesC} \ref{invariant} implies \ref{vector}.
\end{claim}
In fact, take $v\in T_q\Si$ with $(d\pi_{_W})_q(v)=0$. In particular we 
have that $v\in \rm{Ker}(d(\pi_{_W}|_M)_q))$. By (\ref{kernelM}) we have 
that $v\in T_qS'$, hence $v$ is orthogonal to $\g'(0)$. Since the geodesic 
$\g$ intersects $W$ at $p$, we conclude that 
\ref{vector} holds (by taking $\eta=\g'(0)$). Claim \ref{AimpliesC} is proved.

Let $P:T_q(\Q)\to T_p(\Q)$ be
the parallel transport along $\g$. Take $V\subset T_q(\Q)$ such that $T_pW=P(V)$.
\begin{claim}\label{direct} The vector spaces $\real \ga'(0)=\{t\,\ga'(0)\mid t\in \real\}$, $V$
and $T_qS'$ are mutually orthogonal.
\end{claim}
In fact, since $S$ is totally geodesic and $T_pW$ is orthogonal to $T_pS$ it follows that $V=P^{-1}(T_pW)$ is orthogonal
to $T_qS=\real\g'(0)+T_qS'$. And clearly we have that $\g'(0)$ is orthogonal to $T_qS'$.

\begin{claim}\label{AimpliesB} \ref{invariant} implies \ref{orthogonal} if $c\geq 0$.
\end{claim}
In fact, take a unit vector $\eta\in (T_qM)^{\perp}$. From Claim \ref{direct} we may write $\eta=a\g'(0)+\xi+u$, with $a\in\real$, $\xi\in V$ and $u\in T_qS'$. Since $\eta$ is orthogonal
to $M$ and $T_qS'\subset T_qM$ we obtain that $u=0$. If $a=0$ then $\lf<\eta,\g'(0)\rg>=0$, hence
$\g'(0)\in T_qM$ which contradicts Claim \ref{glinha}. Thus we obtain that $a\not=0$. If
$\eta$ and $\g'(0)$ are linearly dependent, then \ref{orthogonal} holds, since
$\g$ intersects $W$, hence we are done in this case.
Thus from now on we may assume  that $\xi\not=0$.

We consider the unique totally geodesic surface $N^2$ of constant curvature $c$ such that $T_q(N^2)$
agrees with the plane generated by $\g'(0)$ and $\xi$. In particular the images of $\g$ and
$\g_\eta$ are contained in $N^2$. By construction we have that $w=P\xi\in T_pW$. Since
$N^2$ is totally geodesic and $\xi\in T_q(N^2)$ it holds that $w=P\xi\in T_p(N^2)$, hence
the image of the geodesic $\g_w$ is contained in $N^2$. If $c>0$, the images of $\g_\eta$
and $\g_w$ must intersect, since they are nontrivial geodesics of the $2$-dimensional
sphere $N^2$, which implies that \ref{orthogonal} holds. If $c=0$ and $\g_\eta$ does not intersect $\g_w$ then they are parallel to
each other. Since $\g_w$ is orthogonal to $\g'(r_0)$ we will have that $\eta$ is orthogonal
to $\g'(0)$ which contradicts the fact that $a\not=0$. This contradiction concludes the proof of Claim \ref{AimpliesB}. Theorem \ref{revolution} is proved.
\end{proof}

The following proposition was mentioned in Remark \ref{c}. 
\begin{proposition}\label{true} Property \ref{vector} is always true if $c>0$.
\end{proposition}
\begin{proof} Fix $q\in\Si$ and $v\in T_q\Si$ with $d(\pi_{_W})_qv=0$. Thus 
it holds that  
$v\in T_q(S_{pW})$, where $p=\pi_{_W}(q)$ (see Lemma \ref{piw}). Let $\g:[0,r_0]\to S_{pW}$ be 
a normal minimizing geodesic from 
$q$ to $p$ satisfying $L(\g)=r_0=d(q,W)$. Fix a unit vector $w\in T_pW$. Let 
$\eta\in T_q(\Q)$ be given by the parallel transport of $w$ along $\g$. Since 
$S_{pW}$ is totally geodesic and $w$ is orthogonal to $T_p(S_{pW})$ we have that $\eta$ is orthogonal to $T_q(S_{pW})$, 
hence it is orthogonal to $v$.  By using again the unique totally geodesic surface $N^2$ such that $T_q(N^2)={\rm span}\{\g'(0),\xi\}$ we obtain that $\g_\eta$ intersects $\g_w$, hence 
it intersects $W$. Proposition \ref{true} is proved.  
\end{proof}
 The next proposition was mentioned in Remark \ref{sharp}.
\begin{proposition} \label{bsubmersion} Let $\Si$ be a hypersurface of $\Q$ with $c\le 0$ and $W\subset \Q$ be
a complete totally geodesic submanifold with $\Si\cap W=\emptyset$. Then the property {\rm \ref{vector}} implies that $\pi_{_W}|_\Sigma$ is a submersion.
\end{proposition}
\begin{proof} Assume by contradiction that \ref{vector} holds and that $\pi_{_W}|_\Sigma$ is not a submersion. Then there exists $q\in \Si$ 
such that $d(\pi_{_W}|_\Si)_q:T_q\Si\to T_pW$ is not surjective.
Consider as above a shortest normal geodesic $\g:[0,r_0]\to \Q$ from $q$ to $W$, namely, assume that
$\g(0)=q$, $\g(r_0)=p=\pr(q)\in W$ and $L(\g)=r_0=d(q,W)$. We consider again the totally geodesic
submanifold $S_{pW}=\pi_W^{-1}(\{p\})$ (see Lemma \ref{piw}).  

 Since $d(\pi_{_W}|_\Si)_q$ is not surjective, it holds that the intersection between $\Si$ and $S_{pW}$ is not
 transversal at $q$. In fact, if $T_q\Si+T_q(S_{pW})=T_q(\Q)$ then we have by Lemma \ref{piw} that
  $(d\pi_{_W})_q(T_q\Si)=(d\pi_{_W})_q(T_q(\Q))=T_pW$, which contradicts the hypothesis 
 that $d(\pi_{_W}|_\Si)_q$ is not surjective.
 
 Since $\Si$ is a hypersurface and it does not intersect $S_{pW}$ transverselly at $q$, 
 we conclude that $T_q(S_{pW})\subset T_q\Si$, hence $\g'(0)\in T_q\Si$. 
 Since $d(\pi_{_W})_q(\g'(0))=0$, Property \ref{vector} implies that there exists a unit vector $\eta$ orthogonal to 
 $\g'(0)$ such that the geodesic $\g_\eta$ intersects $W$. However 
 the facts that $c\le 0$,  $W$ is totally geodesic and $\eta$ is orthogonal 
 to $\g'(0)$ imply together that $\g_\eta$ may not intersect $W$, which give 
 us a contradiction. Proposition \ref{bsubmersion} is proved.
  \end{proof}

\section {\bf Distance function from subsets} \begin{proof}[{\bf Proof of Theorem {\bf\ref{cylinder}}}]
Let $V$ be a neighborhood of a point $x_0$ in $\Si$ such that the restriction $f|_V:V\to M$ is an embedding and denote by $\Si'=f(V)$. Fix $p,q\in \Si'$ and consider a differentiable curve $\al:[a,b]\to \Si'$ with $\al(a)=p$ and $\al(b)=q$ parameterized by arc length. Let $\rho:[a,b]\to \real$ be given by $\rho(s)=G(\al(s))$. By using that $G$ is a Lipschitz function we have that
$$|\rho(s)-\rho(t)|=|G(\al(s))-G(\al(t))|\leq C d(\al(s),\al(t))\leq C L(\al|_{[s,t]})=C |s-t|.$$
Thus, since $\rho$ is a Lipschitz function, it must be differentiable almost everywhere and satisfy the equality $\rho(b)=\rho(a)+\int_a^b\rho'(s)ds$. We fix $s_0\in (a,b)$ such that $\rho'(s_0)$ exists.
\begin{claim}\label{clcylinder} $\rho'(s_0)=0$.
\end{claim}
In fact, by hypothesis, there exists a nontrivial geodesic $\ga:[0,1]\to M$ satisfying
\begin{enumerate} [(i)]
\item $\ga(0)=\al(s_0)$;
\item \label{etaorthogonal} $\ga'(0)$ is orthogonal to $\al'(s_0)$;
\item \label{noponto} $C L(\ga)=|G(\al(s_0))-G(\ga(1))|=|\rho(s_0)-G(\ga(1))|$.
\end{enumerate}
Since $L(\ga)>0$ it follows that $G(\al(s_0))-G(\ga(1))\not=0$. By replacing
$G$ by $-G$ if necessary, we may assume that $G(\al(s_0))-G(\ga(1))>0$. Now we
choose $0<t_0<1$ sufficiently small so that $\al(s_0)$ is contained
in a strongly convex ball $B\subset M$ centered at $\g(t_0)$.
Choose $0<\de<\ep$ sufficiently small so that $I=(s_0-\de,s_0+\de)\subset (a,b)$, $\al([s_0-\de,s_0+\de])\subset B$ and $G(\al(s))-G(\ga(1))>0$ for all $s\in I$. Consider
the smooth map $r:B\to \real$ given by $r(x)=d(\g(t_0),x)$ and the map $h:I\times [0,1]\to M$ given by
\begin{enumerate}[(a)]
\item $h(s,t)=\exp_{\al(s)}\left(\frac{t}{t_0}\left(\exp_{\al(s)}^{-1}\ga(t_0)\right)\right)$, for  $s\in I$ and $t\in [0,t_0]$;
\item $h(s,t)=\ga(t)$, for $s\in I$ and $t\in [t_0,1]$.
\end{enumerate}
Consider the curve $h_s:t\in [0,1]\mapsto h(s,t)$. Note that $L(h_s)=L(\g|_{[t_0,1]})+r(\al(s))$.
Since $\al$ is differentiable we obtain that
\begin{eqnarray}\label{dl} \lf.\fr d{ds}L(h_s)\rg|_{s=s_0}&=&(r\circ\al)'(s_0)=\lf<\na r(\al(s_0)),\al'(s_0)\rg>
\\&=&\lf<-\g'(0),\al'(s_0)\rg>=0.\nonumber
\end{eqnarray}
Since $h_s(0)=\al(s)$ and $h_s(1)=\ga(1)$ we have that
\begin{eqnarray}\label{Lhs}
C L(h_s)&\geq& C d(\al(s),\ga(1))\geq |G(\al(s))-G(\ga(1))|= G(\al(s))-G(\ga(1))\nonumber\\&=&\rho(s)-G(\ga(1)),
\end{eqnarray}
for all $s\in I$.  Thus, using \ref{noponto}, (\ref{dl}) and (\ref{Lhs}), we obtain that
\begin{eqnarray*}
\rho'(s_0)&=&\lim_{s\to s_0 \atop{s>s_0}} \frac{\rho(s)-\rho(s_0)}{s-s_0} \leq \lim_{s\to s_0 \atop{s>s_0}} \frac{C L(h_s)+G(\ga(1))-\rho(s_0))}{s-s_0}\\ &=& \lim_{s\to s_0 \atop{s>s_0}} \frac{C L(h_s)-C L(h_{s_0})}{s-s_0}= C
\left.\frac{d}{ds}\right|_{s=s_0} L(h_s)=0
\end{eqnarray*}
and
\begin{eqnarray*}
\rho'(s_0)&=&\lim_{s\to s_0 \atop{s<s_0}} \frac{\rho(s)-\rho(s_0)}{s-s_0} \geq \lim_{s\to s_0 \atop{s<s_0}} \frac{C L(h_s)+G(\ga(1))-\rho(s_0))}{s-s_0}\\ &=& \lim_{s\to s_0 \atop{s<s_0}} \frac{C L(h_s)-C L(h_{s_0})}{s-s_0}= C
\left.\frac{d}{ds}\right|_{s=s_0} L(h_s)=0
\end{eqnarray*}
Thus it holds that $\rho'(s_0)=0$ and Claim \ref{clcylinder} is proved.

Thus, using that
$\rho(b)=\rho(a)+\int_a^b\rho'(s)ds=\rho(a)$, we obtain that $G|_{\Si'}$ is constant. Since $\Si$ is connected we have that the function $G\circ f$ is constant, which concludes the proof of Theorem \ref{cylinder}.
 \end{proof}

\subsection{Proof of Theorem \ref{hadamard}.}
Let $f:\Si\to M$ be a differentiable immersion of a connected manifold $\Si$ in a Hadamard manifold $M$.
Assume that there exists $x_0\in M({\infty})$ such that for all point $p\in \Si$ and $v\in T_p\Si$ there exists a unit vector $\eta\in T_{f(p)}M$ orthogonal to $df_pv$ satisfying that $\g_\eta(\infty)=x_0$.

Fix a ray $\al:[0,+\infty)\to M$ parameterized by arc length such that $\al(\infty)=x_0$.
We recall that a horosphere $\m H_{x_0}$ of $M$ associated with $x_0\in M(\infty)$ is a level set of the Busemann function $h_\al:M\to \real$ given by
$$h_\al(x)=\lim_{t\to +\infty}d(x,\al(t))-t.$$
It is well known that $h_\al$ is a Lipschitz function with Lipschitz constant $1$ and $h_\al(\al(s))=-s$, for all $s$. Furthermore, if $\g:[0,+\infty)\to M$ is another ray such that $\g(\infty)=x_0$ then the Busemann functions $h_\ga$ and $h_\al$ differ by a constant.

Now fix $p\in \Si$ and $v\in T_p\Si$. By hypothesis there exists a unit vector $\eta\in T_{f(p)}M$ orthogonal to $df_pv$ satisfying that $\g_\eta(\infty)=x_0$. Thus we have that $$h_\al(\ga_\eta(0))-h_\al(\ga_\eta(1))=h_{\ga_{\eta}}(\ga_\eta(0))-h_{\ga_\eta}(\ga_\eta(1))=1=
L({\ga_\eta}|_{[0,1]}).$$
Thus we can apply Theorem \ref{cylinder} with $G=h_\al$ to obtain that $G\circ f$ is constant.
Theorem \ref{hadamard} is proved.

\section{Examples}\label{examples}
The following example (see Figure \ref{fig-annulus}) shows that in the case $c\le 0$ the assumption that $\pi_{_W}|_{\Si}$ is a submersion is essential to obtain that \ref{invariant} implies \ref{vector} (compare with Proposition \ref{true}).
\begin{example}\label{annulus} Let $W$ be a complete totally geodesic submanifold of $\Q$, with $c\leq 0$. Take $S_{pW}=\pi_{_W}^{-1}(p)$, for some $p\in W$. Fix $0\le a<b\le+\infty$ and set $$\Si=\{z\in S_{pW} \bigm| a<d(z,W)<b\}.$$
It is easy to see that $\Si$ is invariant under the action of $G_W$, hence it satisfies \ref{invariant}. Fix $q\in \Si$ and $v=\g'(0)$, where $\g$ is the normal geodesic from $q$ to $p$. For any vector $\eta\in T_q(\Q)$ orthogonal to $v$, the geodesic $\g_\eta$ does not intersect $W$, hence \ref{vector} does not hold.
\end{example}

\begin{figure}
\scalebox{0.7} 
{
\begin{pspicture}(0,-2.39)(10.02,2.39)
\definecolor{color365b}{rgb}{0.8,0.8,0.8}
\psellipse[linewidth=0.04,dimen=outer,fillstyle=solid,fillcolor=color365b](5.01,-0.01)(5.01,1.12)
\psellipse[linewidth=0.04,dimen=outer,fillstyle=gradient,gradlines=2000,gradbegin=white,gradend=white,gradmidpoint=1.0](5.01,-0.016666668)(3.07,0.42518517)
\psdots[dotsize=0.12](4.96,-0.01)
\usefont{T1}{ptm}{m}{n}
\rput(0.50145507,0.015){$\Si$}
\psline[linewidth=0.04cm](4.96,-0.41)(4.96,2.37)
\psline[linewidth=0.04cm,linestyle=dashed,dash=0.16cm 0.16cm](4.96,-0.35)(4.96,-1.07)
\psline[linewidth=0.04cm](4.96,-1.01)(4.96,-2.37)
\usefont{T1}{ptm}{m}{n}
\rput(5.3114552,2.075){$W$}
\psline[linewidth=0.04cm,arrowsize=0.05291667cm 2.0,arrowlength=1.4,arrowinset=0.4]{<-}(8.24,-0.01)(9.34,-0.01)
\usefont{T1}{ptm}{m}{n}
\rput(8.291455,0.195){$v$}
\usefont{T1}{ptm}{m}{n}
\rput(9.291455,-0.225){$q$}
\psdots[dotsize=0.12](9.32,-0.01)
\end{pspicture}}
\caption{Referred in Example \ref{annulus}.}\label{fig-annulus}
\end{figure}
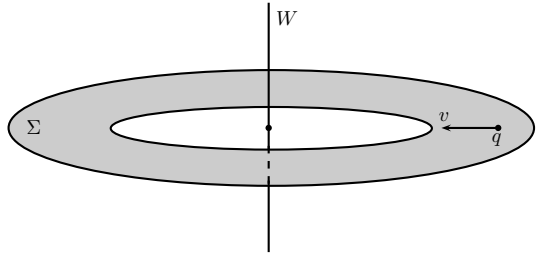

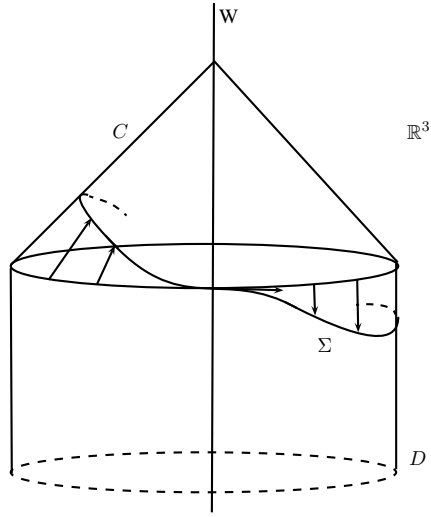
\begin{figure}
\scalebox{0.7} 
{
\begin{pspicture}(0,-4.86)(8.48335,4.86)
\psline[linewidth=0.04cm](3.88,4.84)(3.84,-4.84)
\psellipse[linewidth=0.04,dimen=outer](3.7,-0.15835927)(3.7,0.44)
\psline[linewidth=0.04cm](7.34,-0.038359273)(7.34,-4.018359)
\usefont{T1}{ptm}{m}{n}
\rput(4.1622267,4.586641){W}
\usefont{T1}{ptm}{m}{n}
\rput(7.77291,2.4266407){$\real^3$}
\psline[linewidth=0.04cm](3.88,3.74)(7.37,-0.088359274)
\psline[linewidth=0.04cm](3.8799999,3.7216406)(0.02000005,-0.13835932)
\psellipse[linewidth=0.04,linestyle=dashed,dash=0.16cm 0.16cm,dimen=outer](3.68,-4.078359)(3.68,0.4)
\psbezier[linewidth=0.04](5.22,-0.8383593)(5.88,-1.18)(7.38,-1.9383593)(7.38,-1.123265)
\psbezier[linewidth=0.04,linestyle=dashed,dash=0.16cm 0.16cm](6.58,-0.89798754)(6.8194118,-0.86)(7.32,-0.8878758)(7.32,-1.08)
\psbezier[linewidth=0.04](5.46,-0.96)(4.0,-0.14)(3.26,-1.22)(1.82,0.42)
\psbezier[linewidth=0.04,linestyle=dashed,dash=0.16cm 0.16cm](2.22,0.8)(1.94,1.06)(1.84,1.06)(1.5108463,1.1691537)
\psbezier[linewidth=0.04](2.06,0.18)(1.51,0.69164073)(1.09,1.3116407)(1.49,1.1916407)
\usefont{T1}{ptm}{m}{n}
\rput(2.10291,2.4066405){$C$}
\usefont{T1}{ptm}{m}{n}
\rput(7.75291,-3.8133593){$D$}
\usefont{T1}{ptm}{m}{n}
\rput(5.98291,-1.6533593){$\Si$}
\psline[linewidth=0.04cm,arrowsize=0.05291667cm 2.0,arrowlength=1.4,arrowinset=0.4]{<-}(5.2,-0.62)(3.86,-0.58)
\psline[linewidth=0.04cm](0.02,-0.19835937)(0.04,-4.098359)
\psline[linewidth=0.04cm,arrowsize=0.05291667cm 2.0,arrowlength=1.4,arrowinset=0.4]{<-}(1.56,0.76)(0.74,-0.42)
\psline[linewidth=0.04cm,arrowsize=0.05291667cm 2.0,arrowlength=1.4,arrowinset=0.4]{<-}(2.0,0.24)(1.66,-0.5)
\psline[linewidth=0.04cm,arrowsize=0.05291667cm 2.0,arrowlength=1.4,arrowinset=0.4]{<-}(5.8,-1.1)(5.78,-0.5)
\psline[linewidth=0.04cm,arrowsize=0.05291667cm 2.0,arrowlength=1.4,arrowinset=0.4]{<-}(6.62,-1.42)(6.6,-0.42)
\end{pspicture} 
}
\caption{Referred in Example \ref{submersion}}\label{cone}
\end{figure}

According to Theorem \ref{revolution} we have that \ref{orthogonal} implies \ref{invariant} for any space form of constant curvature $c\in\real$. However, the next example shows that without the condition that $\pi_{_W}|_\Si:\Si\to W$ is
a submersion this implication may fail.
\begin{example} \label{submersion} Consider the cone and cylinder given, respectively, by $$C=\{(x,y,z)\in\real^3\bigm| 0\le z< 1, (z-1)^2=x^2+y^2\}$$ and  $$D=\{(x,y,z)\in\real^3\bigm|z\le 0, x^2+y^2=1\},$$ and let $\Si, W\subset \real^3$ be as in Figure \ref{cone}.
More precisely, consider smooth functions $\mu,\nu:(-\ep,\ep)\to[0,+\infty)$ for some small $\ep>0$, satisfying that
\begin{equation*}\left\{
\begin{array}{ll}
\mu(t)=0 \mbox{ for all } t\le 0; & \mu(t)>0 \mbox{ for all } t>0;\\
\nu(t)>0 \mbox{ for all } t\le 0; & \nu(t)=0 \mbox{ for all } t>0.
\end{array}\right.
\end{equation*}
Consider the curve $\al(t)=(\cos(t),\sin(t),0)$, with $t\in\real$. Let $\be:(-\ep,\ep)\to \real^3$ be the smooth curve given by $$\be(t)=\al(t)+\nu(t)(-\cos(t),-\sin(t),1)+\mu(t)(0,0,-1).$$
Let $\Si$ be the image of $\be$ and $W$ the $z$-axis. It is easy to see that $\Si$ is a smooth embedded submanifold if $\ep$ is sufficiently small. We have that $\be(t)$ belongs to
 the cone $C$ if $t<0$ and to the cylinder $D$ if $t\ge 0$. Thus it is not difficult to see that $\Si$ satisfies \ref{orthogonal} in Theorem \ref{revolution}. Note that any submanifold containing $\Si$ and invariant under the $G_W$ action
should contain an open neighborhood of $\al(0)$ in the non-smooth continuous hypersurface $C\cup D$,
which implies that $\Si$ does not satisfy \ref{invariant}.
Note that $\pi_{_W}|_\Si$ is not a submersion at the point $\be(0)$, since
$$d(\pi_{_W})_{\be(0)}\be'(0)=\pi_{_W}(\be'(0))=\pi_{_W}(0,1,0)=0.$$
\end{example}
{
The following example shows that Theorem \ref{revolution} is sharp in the sense that
\ref{invariant}
does not imply $\ref{orthogonal}$ in the case $c<0$ (see Remark \ref{sharp} in the Introduction).

\begin{example}\label{examplehyperbolic} Consider the hyperbolic space $\mathbb{H}^3$ in
the half space model $\real^3_+=\{(x, y, z)\bigm|z>0\}$. Let $W=\{(0,0,z)\bigm|z>0\}$ be a  vertical (totally geodesic) line in $\mathbb{H}^3$. Let
$\Si=\{(x,y,z)\bigm|x^2+y^2=1, z>0\}\subset \hy^n$ be the cylinder of axis $W$ and Euclidean radius $1$.
We first verify that $\pi_{_W}|_\Si$ is a submersion. For this we take $q=(x,y,z)\in \Si$ and
the curve
 $\al:(0,+\infty)\to \Si$ given by $\al(t)=(x,y,t)$. We have that $\al(z)=q$ and
 $\al'(z)=(0,0,1)$. Set $\be:(0,+\infty)\to W$ given by $\be(t)=(0,0,\sqrt{1+t^2})$.
 It is easy to see that $\be(t)=\pi_{_W}(\al(t))$, hence we obtain that
 $$(d\pi_{_W}|_\Si)_q(\al'(z))=\be'(z)=\lf(0,0,\fr{z}{\sqrt{1+z^2}}\rg)\not=0,$$
hence we have that $(d\pi_{_W}|_\Si)_q:T_q\Si\to T_{\pi_{_W}(q)}W$ is surjective and
$\pi_{_W}|_\Si$ is a submersion. Since $\Si$ is a hypersurface invariant under rotations around $W$
we see that $\Si$ satisfies \ref{invariant}. Now we will verify that $\Si$ does
not satisfy \ref{orthogonal}. We choose $q=(x,y,z)\in \Si$ with $0<z\le 1$. For any unit vector $\eta$ orthogonal to $\Si$ the geodesic $\g_\eta$ will not intersect $W$ since it is contained in the Euclidean sphere
of center $(x,y,0)$ and radius $z$ (see Figure \ref{fig-cylinder}). This shows that \ref{invariant}
does not imply $\ref{orthogonal}$ in the case $c<0$.
\end{example}}
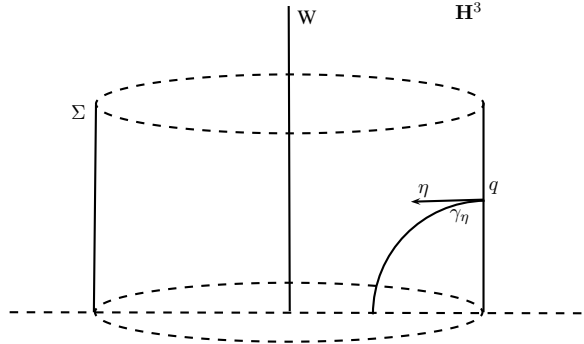
\begin{figure}[!b]
\scalebox{0.7} 
{
\begin{pspicture}(0,-2.4191992)(10.88,4.0791993)
\psline[linewidth=0.04cm](5.3,3.9808009)(5.32,-1.8191992)
\psline[linewidth=0.04cm,linestyle=dashed,dash=0.16cm 0.16cm](0.0,-1.8391992)(10.86,-1.8591992)
\psellipse[linewidth=0.04,linestyle=dashed,dash=0.16cm 0.16cm,dimen=outer](5.32,2.1208007)(3.7,0.58)
\psline[linewidth=0.04cm](9.0,2.1408007)(9.0,-1.8391992)
\psline[linewidth=0.04cm](1.64,2.1008008)(1.6,-1.8391992)
\usefont{T1}{ptm}{m}{n}
\rput(5.6461134,3.7658007){W}
\usefont{T1}{ptm}{m}{n}
\rput(1.3114551,1.9658008){$\Sigma$}
\psarc[linewidth=0.04](9.06,-1.8791993){2.16}{91.02303}{180.0}
\psline[linewidth=0.04cm,arrowsize=0.05291667cm 2.0,arrowlength=1.4,arrowinset=0.4]{->}(9.0,0.30080077)(7.62,0.26080078)
\usefont{T1}{ptm}{m}{n}
\rput(7.861455,0.4658008){$\eta$}
\usefont{T1}{ptm}{m}{n}
\rput(8.721455,3.8858008){$\hy^3$}
\usefont{T1}{ptm}{m}{n}
\rput(9.191455,0.52580076){$q$}
\usefont{T1}{ptm}{m}{n}
\rput(8.5514555,-0.014199219){$\ga_\eta$}
\psellipse[linewidth=0.04,linestyle=dashed,dash=0.16cm 0.16cm,dimen=outer](5.3,-1.8391992)(3.7,0.58)
\end{pspicture} 
}
\caption{Reffered in Example \ref{examplehyperbolic}}\label{fig-cylinder}
\end{figure}

The next example shows that Theorem \ref{revolution} may not be improved to obtain that \ref{vector} implies  \ref{invariant} in
the case $c>0$ (see Remarks \ref{c}, \ref{sharp}).
\begin{example}\label{example sphere} Consider the standard unit sphere $S^3$ and  the natural totally geodesic inclusion $S^2\subset S^3$.  Consider
on $S^2$ the image $W$ of a closed geodesic on $S^2$ (see Figure \ref{fig-sphere}). Let $\Si$ be an open subset   of
$S^2$ satisfying that $\Si\cap\{W\cup V_W\}=\emptyset$. Clearly we have
that $\pi_{_W}|_\Si:\Si\to W$ is a submersion. First we will see that
$\Si$ satisfies \ref{vector}. In fact,
fix a point
$q$ on $\Si$ and any unit vector $v\in T_p\Si$. Choose a unit vector $\eta\in T_qS^2$ orthogonal to $v$. The geodesic $\g_\eta$ must
remain contained in $S^2$, hence it will intersect $W$ and \ref{vector} holds.
Now we will see that \ref{invariant} does not hold. We observe that, since $\Si$ is an open subset of $S^2$, the union of orbits $\m V=\cup_{x\in\Si} G_W(x)$ is an open
subset of $S^3$. Thus any submanifold $M$ containing $\Si$ and invariant under the
action of $G_W$ must contain $\m V$, hence $M$ may not be a hypersurface. We conclude that $\Si$ does not satisfy \ref{invariant}.
\begin{figure}
\centering
\scalebox{0.8} 
{
\begin{pspicture}(0,-2.75)(6.6818943,2.75)
\definecolor{color117b}{rgb}{0.8,0.8,0.8}
\pscircle[linewidth=0.04,dimen=outer](2.74,0.01){2.74}
\psellipse[linewidth=0.032,dimen=outer](2.78,0.0)(0.4,2.75)
\psellipse[linewidth=0.02,linestyle=dashed,dash=0.16cm 0.16cm,dimen=outer,fillstyle=solid,fillcolor=color117b](4.7,-0.09)(0.66,0.64)
\psdots[dotsize=0.08](4.78,-0.03)
\psline[linewidth=0.032cm,arrowsize=0.05291667cm 2.0,arrowlength=1.4,arrowinset=0.4]{<-}(4.02,0.41)(4.8,-0.03)
\psline[linewidth=0.032cm,arrowsize=0.05291667cm 2.0,arrowlength=1.4,arrowinset=0.4]{<-}(4.26,-0.95)(4.8,-0.01)
\usefont{T1}{ptm}{m}{n}
\rput(3.9114552,0.255){$v$}
\usefont{T1}{ptm}{m}{n}
\rput(4.5614552,0.795){$\Sigma \subset S^2$}
\usefont{T1}{ptm}{m}{n}
\rput(5.371455,2.435){$S^2\subset S^3$}
\usefont{T1}{ptm}{m}{n}
\rput(4.5014553,-0.965){$\eta$}
\psbezier[linewidth=0.02](4.7826085,-0.03)(4.1,-1.25)(3.3,-1.35)(2.52,-1.57)
\usefont{T1}{ptm}{m}{n}
\rput(2.071455,0.115){$W$}
\usefont{T1}{ptm}{m}{n}
\rput(4.951455,-0.185){$q$}
\end{pspicture} 
}
\caption{Reffered in Example \ref{example sphere}}\label{fig-sphere}
\end{figure}
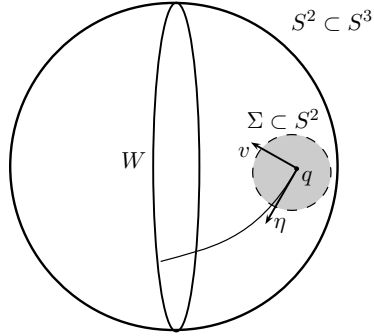
\end{example}

The example below presents a nontrivial situation where Theorem \ref{revolution} applies. 
\begin{example}\label{complex} Consider the map $f:\real^2-\{(0,0)\}\to \real^4$ given by
$f(x,y)=(x,y,e^x\cos y, e^x\sin y)$ and let $\Si$ be the image of $f$. Set $W=\{(0,0)\}\times\real^2$ and
consider the natural projection $\pi_{_W}:\real^4\to W$.
We claim that $\Si$ and $W$ satisfy the hypotheses of Theorem \ref{revolution}, and that any
plane orthogonal to $W$ at a point $p\in W-\{(0,0,1,0), (0,0,0,0)\}$ intersects $\Si$ in infinitely many
isolated points (see Remark \ref{sharp}). In fact, we first note that $\Si\cap W=\emptyset$.
We have that
\begin{equation}\label{vectors}
\fr{\p f}{\p x}=(1,0,\, e^x\cos y,\, e^x\sin y) \ \mbox{ and } \ \fr{\p f}{\p y}=(0,1,\, -e^x\sin y,\, e^x\cos y),
\end{equation}
hence $f$ is an immersion. Since $\Si$ is a
smooth graph we conclude that $\Si$ is a smooth embedded submanifold. The vectors
$$\pi_{_W}\lf(\fr{\p f}{\p x}\rg)=(0,0,\, e^x\cos y,\, e^x\sin y),\ \pi_{_W}\lf(\fr{\p f}{\p y}\rg)=(0,0,\, -e^x\sin y,\, e^x\cos y)$$
are linearly independent, hence $\pi_{_W}|_{\Si}:\Si\to W$ is a submersion. Now we will see that
Item \ref{orthogonal} in Theorem \ref{revolution} is satisfied. To obtain this it suffices to prove that
 $\lf(q+\lf(T_q\Si\rg)^\perp\rg)\cap W\not=\emptyset$, for any $q=f(x,y)\in\Si$. By a simple computation
 using (\ref{vectors}) we obtain that
 $$\lf(T_q\Si\rg)^\perp=\{(-c\,e^x\cos y-d\,e^x\sin y,c\,e^x\sin y-d\,e^x\cos y,\,c,\,d)\bigm|
c,d\in\real\}.$$
Thus we have $\lf(q+\lf(T_q\Si\rg)^\perp\rg)\cap W\not=\emptyset$ if and only if
the linear system
\begin{equation*}\left\{
\begin{array}{l}
x=c\,e^x\cos y + d\,e^x\sin y,\\
y=-c\, e^x\sin y + d\, e^x\cos y
\end{array}\right.
\end{equation*}
has a solution, and this is the case. Now, take $$p=(0,0,\al,\be)\in W-\{(0,0,1,0),(0,0,0,0)\}.$$ We will see that the plane $S_{pW}=p+W^\perp$ intersects $\Si$ at infinitely many isolated points. To see this, note that $S_{pW}=\{(u,v,\al,\be)\bigm| u,v \in \real\}.$ Thus an easy computation shows that $$S_{pW}\cap \Si=\left\{(\log(\sqrt{\al^2+\be^2}),\te+2k\pi,\al,\be)\bigm| k\in \mathbb{Z}\right\},$$ where $\te$ is any angle satisfying $\cos \te=\frac{\al}{\sqrt{\al^2+\be^2}}$ and $\sin \te=\frac{\be}{\sqrt{\al^2+\be^2}}$. Our claim is proved.
\end{example}

\end{document}